\documentclass[12pt,oneside,reqno]{amsart}
\usepackage[all]{xy}
\usepackage{amsfonts,amsmath,oldgerm,amssymb,amscd}
\UseComputerModernTips
\numberwithin{equation}{section}
\usepackage[breaklinks]{hyperref}
\allowdisplaybreaks

\usepackage{enumerate}

\usepackage{caption} 
\newtheorem{theorem}{Theorem}[section]
\newtheorem{proposition}[subsection]{\bf Proposition}

\newtheorem{lemma}[subsection]{{\bf Lemma}}

\newtheorem{remark}[subsection]{Remark}

\newcommand{\al}{\alpha}

\newcommand{\Z}{\mbox{$\mathbb Z$}}
\newcommand{\Q}{\mbox{$\mathbb Q$}}

\newcommand{\e}{\varepsilon}

\begin{document}

\title[Sums of $S$-units in sum of terms of recurrence sequences]{Sums of $S$-units in sum of terms of recurrence sequences} 

\author[P. K. Bhoi]{P. K.  Bhoi}
\address{Pritam Kumar Bhoi, Department of Mathematics, National Institute of Technology Rourkela-769008, India}
\email{pritam.bhoi@gmail.com}

\author[S. S. Rout]{S. S. Rout}
\address{Sudhansu Sekhar Rout, Department of Mathematics, National Institute of Technology Calicut- 673 601, 
Kozhikode, India.}
\email{sudhansu@nitc.ac.in; lbs.sudhansu@gmail.com}

\author[G. K. Panda]{G. K. Panda}
\address{Gopal Krishna Panda, Department of Mathematics, National Institute of Technology Rourkela-769008, India}
\email{gkpanda\_nit@rediffmail.com}

\thanks{2010 Mathematics Subject Classification: Primary 11B37, Secondary 11D61, 11J86. \\
Keywords: Linear recurrence sequence,  sums of $S$-units, perfect power, Baker's method}
\maketitle
\pagenumbering{arabic}
\pagestyle{headings}

\begin{abstract}
Let $S := \{p_1,\ldots ,p_{\ell}\}$ be a finite set of primes and denote by $\mathcal{U}_S$ the set of all rational integers whose prime factors are all in $S$. Let $(U_n)_{n\geq 0}$ be a non-degenerate linear recurrence sequence with order at least two. In this paper, we provide a finiteness result for the solutions of the Diophantine equation $aU_n + bU_m = z_1 +\cdots +z_r,$ where $n\geq m$ and $z_1, \ldots, z_r\in \mathcal{U}_S$.  
\end{abstract}

\section{Introduction}

For a given positive integer  $k$, let $(U_{n})_{n \geq 0}$ be a linear recurrence sequence  of order $k$ defined by
\begin{equation}\label{eq4}
U_{n} = a_1U_{n-1} + \dots +a_kU_{n-k},
\end{equation}
where $a_1,\dots, a_k \in \Z$ with $a_k\neq 0$ and $U_0,\dots,U_{k-1}$ are integers not all zero. The characteristic polynomial of $U_n$ is given by
\begin{equation}\label{eq5}
f(x):= x^k - a_1x^{k-1}-\dots-a_k = \prod_{i =1}^{t}(x - \alpha_i)^{m_i}\in \Z[X],
\end{equation}
where $\alpha_1,\dots,\alpha_t$ are distinct algebraic numbers and $m_1,\dots, m_t$ are positive integers. It is well known that (see \cite[Theorem C1]{st}) for $n\geq 0$,
\begin{equation}\label{eq6}
U_n=f_1(n)\alpha_1^n +\cdots + f_t(n)\alpha_{t}^n.
\end{equation}
Here $f_1, \ldots, f_t$ are nonzero polynomials in $n$ with degrees less than $m_1, \ldots, m_t$  respectively with coefficients from $\Q(\alpha_1, \ldots, \alpha_t)$. The sequence $(U_n)_{n \geq 0}$ is called {\it simple} if $t=k$ and called {\it degenerate} if there are integers $i, j$ with $1\leq i< j\leq t$ such that $\alpha_i/\alpha_j$ is a root of unity; otherwise it is called {\it non-degenerate}. 

Let $\{p_1,\ldots,p_{\ell}\}$ be distinct primes and put $S = \{p_1,\ldots, p_\ell\}$. Then a rational integer $z$ is an $S$-unit if $z$ can be written as
\begin{equation}\label{eq0}
z = \pm p_{1}^{e_1}\cdots p_{\ell}^{e_{\ell}}
\end{equation}
where $e_1, \ldots, e_{\ell}$ are nonnegative integers and we denote the set of $S$-units by $\mathcal{U}_S$. 

The aim of this paper is to study the Diophantine equation 
\begin{equation}\label{eq8}
aU_n + bU_m = z_1 + \cdots + z_r
\end{equation} 
where $z_1, \ldots, z_r\in \mathcal{U}_S$ and $a$, $b$ are positive integers.  This will serve as an extension of the earlier result by B\'erczes et al. \cite[Theorem 2.2]{bhpr1}.  Several authors have investigated such Diophantine equations combining both $S$-units and recurrence sequences. In this context, we want to mention few references related to these problems (see \cite{bhpr1, bhpr2, sl}) and other references therein.  Further, Peth\H{o} and Tichy  \cite{pt} proved that if $p$ is a fixed prime, then there are only finitely many Fibonacci numbers of the form $p^a + p^b +p^c$ with integers $a>b>c\geq 0$. Bert\'ok et al.  \cite[Theorem 2.1]{bhpr2},  solved  equations of the form $U_n = 2^a + 3^b + 5^c$ completely, when $U_n$ is one of the Fibonacci, Lucas, Pell and associated Pell sequences, respectively.  Also, Hajdu et al. \cite{hs} and Erazo et al. \cite{egl} investigated the problem of $S$-units in the $x$ co-ordinate of solutions of Pell equations. For other related problems concerning $S$-units and recurrence sequences, one can go through (\cite{ev, egybook, egyst}). 

For a given recurrence sequence $(U_n)_{n\geq 0}$, several authors have also studied the problem of finding $(n, m, z)$ such that
\begin{equation}\label{eq2}
U_{n} + U_{m} = 2^{z}.
\end{equation}
In particular, Bravo and Luca studied the case when $U_{n}$ is the Fibonacci sequence \cite[Theorem 2]{BL2015} and the Lucas number sequence \cite[Theorem 2]{Bravo2014}. Recently, Pink and Ziegler \cite[Theorem 1]{Pink2016} have generalized the results of Bravo and Luca \cite{Bravo2014,BL2015}, and considered a more general Diophantine equation 
\begin{equation}\label{eq2a}
U_{n} + U_{m} = wp_{1}^{a_{1}} \cdots p_{s}^{a_{s}}
\end{equation}
in nonnegative integer unknowns $n, m, a_{1}, \ldots, a_{s}$, where $(U_{n})_{n \geq 0}$ is a binary non-degenerate recurrence sequence, $p_{1}, \ldots, p_{s}$ are distinct primes and $w$ is a nonzero integer with $p_{i} \nmid w$ for $1\leq i \leq s$. They proved that under certain assumptions, \eqref{eq2a} has finitely many solutions using lower bounds for linear forms of $p$-adic logarithms. Further, certain variations of the Diophantine equation \eqref{eq8} has also been studied (see e.g. \cite{mr2019, mr}) for binary recurrence sequences.

We are interested in linear recurrence sequences $(U_n)_{n=0}^{\infty}$ with $U_n$ defined as in \eqref{eq6} for which $f_1(n)$ is a nonzero constant (say $\eta_1$). Thus,
\begin{equation}\label{eq7}
U_n=\eta_1\alpha_1^n +f_2(n)\alpha_2^n+\cdots + f_t(n)\alpha_{t}^n.
\end{equation}
Further, we set 

\begin{equation}\label{eq707}
\gamma:=\max\left\{\max_{1 \leq i \leq k}|a_i|, \max_{0 \leq j \leq {k-1}}|U_j| \right\}.
\end{equation}

Now we are ready to state our main result, which extend an earlier result by B\'erczes et al. \cite[Theorem 2.2]{bhpr1}.
 \begin{theorem}\label{th1}
Let $k\geq 2$ be an integer and $a, b$ are positive integers. Let $(U_n)_{n\geq 0}$ be a non-degenerate linear recurrence sequence of order $k$ and let $U_n$ satisfy \eqref{eq7}. Furthermore, assume that the dominant root $\alpha_1$ is not an integer larger than one and the following condition does not hold
\begin{equation} \label{eq:assumption1}
	|U_n|\geq  |aU_n+bU_m|.
\end{equation}
Let $\e>0$ be arbitrary. Then for all $r\geq 1$ and for all solutions $(n, m, z_1,\dots, z_r)$ of \eqref{eq8}
satisfying $|z_i|^{1+\e}<|z_r|$ for $i=1,\dots,r-1$, we have
$$
\max(n,m, |z_1|,\dots,|z_r|)\leq C,
$$
where $C$ is an effectively computable constant depending only on $\e$, $\gamma$, $k$, $\ell$, $p_1,\dots,p_{\ell}$, $r$, $a$, $b$.
\end{theorem}

\begin{remark}
Suppose $U_n = 2^n -1$ and $S= \{2\}$. Then 
\[U_n + U_m = 2^{u} + 2^{v} - 2\] has infinitely many solutions given by $ n = u, m = v$. This shows that the assumption that the dominant root of $U_n$ is not an integer is necessary.
\end{remark}

 \section{Auxiliary results}

To achieve the aforesaid goal we need the following results.

\begin{lemma} [\cite{bhpr1}]\label{lem4}
Let $(U_n)_{n\geq 0}$ be a non-degenerate linear recurrence sequence of order $k \geq 2$.
\begin{enumerate}
\item[(i)] With the notation in \eqref{eq6}, write $$f_i(n)=\beta_{i,0}+\beta_{i,1}n+ \dots + \beta_{i,m_i-1}n^{m_i-1}\ (i=1,\dots,t).$$ Then we have
$$
\max\limits_{1\leq i \leq t,\ 0\leq \ell \leq m_i-1}\left\{h(\alpha_i),h(\beta_{i,\ell})\right\}\leq c_1.
$$
Here $c_1$ is an effectively computable constant depending only on $\gamma$ and $k$.

\vskip.2cm

\item[(ii)] If $n\geq 1$ and $f_i(n)\neq 0$ then
$$
c_2\leq |f_i(n)| \leq c_3n^{m_i-1}\ \ \ (1 \leq i \leq t),
$$
where $c_2$ and $c_3$ are effectively computable constants depending only on $\gamma$ and $k$.

\vskip.2cm

\item[(iii)] Suppose that $\alpha_1$ is a dominant root of the sequence $(U_n)_{n\geq 0}$. Then we have
    $$
    |U_n|\leq c_4n^{k-1}|\alpha_1|^n\ \ \ (n\geq 1),
    $$
where $c_4$ is an effectively computable constant depending only on $\gamma$ and $k$.
\end{enumerate}
\end{lemma}

Let $\psi$ be an algebraic number of degree $d$ with minimal polynomial 
\[c_{0}x^d + c_1x^{d-1} + \cdots + c_d = c_0 \prod_{i=1}^{d}\left(X - \psi^{(i)}\right),\]
where $c_0$ is the leading coefficient of the minimal polynomial of $\psi$ over $\mathbb{Z}$ and the
$\psi^{(i)}$'s are conjugates of $\psi$ in $\mathbb{C}$.  The absolute {\it logarithmic height} of an algebraic number $\psi$ as 
\[
h(\psi) = \frac{1}{d} \left( \log |c_0| + \sum_{i=1}^d \log \max ( 1, |\psi^{(i)}| ) \right). 
\]
In particular, if $\psi = p/q$ is a rational number with $\gcd(p, q) = 1$ and $q >0$, then $h(\psi) = \log \max\{|p|, q\}$.

\begin{lemma}[Matveev \cite{Mat2000}]
\label{lem:matveev}
Denote by $\psi_1,\ldots,\psi_m$ algebraic numbers, not $0$ or $1$, by $\log\psi_1,\ldots,\log\psi_m$ the principal values of their logarithms, by $D$ the degree of the number field $\mathbb{K} = \Q(\psi_1,\ldots,\psi_m)$ over $\Q$, and by $b_1,\ldots,b_m$ rational integers. Define
$$
B=\max\{|b_1|,\ldots,|b_m|\},
$$
and
$$
A_i= \max\{D h(\psi_i),|\log\psi_i|, 0.16\} \quad (1\le i\le m),
$$
where $h(\psi)$ denotes the absolute logarithmic height of $\psi$. Consider the linear form
$$
\Lambda=b_1\log\psi_1+\dots+b_m\log\psi_m
$$
and assume that $\Lambda \ne 0$. Then
\[\log |\Lambda|\geq -C(m,\varkappa) D^2 A_1\cdots A_m \log (eD) \log (eB),\]
where $\varkappa=1$ if $\mathbb K \subset \mathbb{R}$ and $\varkappa =2$ otherwise and
\[C(m,\varkappa)=\min \left\{ \frac 1{\varkappa} \left( \frac 12 em \right) ^{\varkappa} 30^{m+3} m^{3.5},
2^{6m +20} \right\}. \]
\end{lemma}

\begin{lemma}\label{ubound}
Suppose that $(U_n)_{n\geq 0}$ has a dominant root $\alpha_1$. Let $(n, m, z_1,\dots, z_r)$ be a solution of \eqref{eq8} such that $n\geq m$ and
\begin{equation}
\label{ujeq}
|z_i|^{1+\e}\leq |z_r|, \,\,1\leq i \leq r-1
\end{equation}
with some $\e>0$. Then $\log |z_r|<c_7n+c_8$, where $c_6$ and $c_7$ are effectively computable constants depending only on $\e$, $\gamma$, $r$, $k$, $a$ and $b$.
\end{lemma}

\begin{proof}
Note that if $|z_r|^{\frac{\e}{1+\e}}\leq r$, then $\log |z_r|< c_6$, where $c_6$ is an effectively computable constant depending on $\e, r$. Now we may assume that $|z_r|^{\frac{\e}{1+\e}}>r$. 
In view of \eqref{eq8} and \eqref{ujeq}, we have 
\[|z_r|\leq |aU_n| +|bU_m| + \sum_{i=1}^{r-1}|z_i| \leq  |aU_n| +|bU_m|+(r-1)|z_r|^{\frac{1}{1+\e}}.\]
Since $\alpha_1$ is dominant, using Lemma \ref{lem4}(iii), we deduce
\begin{equation}\label{eq708}
|z_r|^{\frac{1}{1+\e}}\left(|z_r|^{\frac{\e}{1+\e}}- (r-1)\right)\leq c_5 n^{k-1}|\alpha_1|^n.
\end{equation}
Since $|z_r|^{\frac{\e}{1+\e}}>r$, then from \eqref{eq708}, we have
\[\log |z_r| < c_7 n +c_8.\]
\end{proof}

\begin{proposition}\label{prop1}
Let $(U_n)_{n\geq 0}$ be a non-degenerate linear recurrence sequence of order $k \geq 2$ and let $U_n$ satisfy \eqref{eq7}. Furthermore, assume that the dominant root $\alpha_1$ is not an integer larger than one and 
\begin{equation} \label{eq:assumption}
 |aU_n+bU_m| > \frac{1}{2}|\eta_1|\alpha_1^n.
\end{equation} Let $\e>0$ be arbitrary and $n\geq m$. Then for all $r\geq 1$ and for all solutions $(n, m, z_1,\dots, z_r)$ of \eqref{eq8}
satisfying $|z_i|^{1+\e}<|z_r|$ for $i=1,\dots,r-1$, we have
$$
n-m \leq C \log n,
$$
where $C$ is an effectively computable constant depending only on $\e$, $\gamma$, $k$, $\ell$, $p_1,\dots,p_{\ell}$, $r$, $a$, $b$.
\end{proposition}
\begin{proof}
 Let $c_{9},  \ldots, c_{18}$ be positive numbers which are effectively computable  in terms of $\e, \gamma, k, \ell, p_1,\dots, p_\ell, r, a, b$. Using \eqref{eq7}, we rewrite \eqref{eq8} as
$$
a\eta_1\alpha_1^{n} - z_r  = \sum_{i=1}^{r-1}{z_i} - a\sum_{j=2}^{t}f_j(n)\alpha_j^n - b\eta_1\alpha_1^m -b\sum_{j=2}^{t}f_j(m)\alpha_j^m.
$$
Dividing both sides of the above equation by $z_r$, we get
\[a\eta_1\alpha_1^{n}z_r^{-1}-1= \frac{\sum_{i=1}^{r-1}{z_i} }{z_r} -  \frac{ a\sum_{j=2}^{t}f_j(n)\alpha_j^n}{z_r} -  \frac{b\eta_1\alpha_1^m+b\sum_{j=2}^{t}f_j(m)\alpha_j^m }{z_r},\]
and putting $\Phi_1:=a\eta_1\alpha_1^{n}z_r^{-1}$, we obtain
\begin{equation}\label{eq:thm-9}
\left|\Phi_1 -1 \right| \leq \frac{\sum_{i=1}^{r-1}{|z_i|}}{|z_r|}+\frac{a\sum_{j=2}^{t}|f_j(n)||\alpha_j|^n}{|z_r|} + \frac{b|\eta_1|\alpha_1^m+b\sum_{j=2}^{t}|f_j(m)||\alpha_j|^m}{|z_r|}.
\end{equation}
Firstly, we derive an upper bound for $\frac{b|\eta_1|\alpha_1^m+b\sum_{j=2}^{t}|f_j(m)||\alpha_j|^m}{|z_r|}$. Since, by assumption, for every $1 \leq i \leq r-1$ we have $|z_i| \leq |z_r|^{\frac{1}{1+\e}}$, from
\eqref{eq8} and \eqref{eq:assumption} we get
\begin{equation}\label{eq:thm-14}
|z_r|>\frac{1}{2r}|\eta_1|\alpha_1^n.
\end{equation}
Using Eq.\eqref{eq:thm-14} and Lemma \ref{lem4}(ii), we may write 
\begin{equation}\label{eq:thm-17}
\frac{b|\eta_1|\alpha_1^m+b\sum_{j=2}^{t}|f_j(m)||\alpha_j|^m}{|z_r|}<\frac{2rtbc_3m^{k-1}\alpha_1^m}{|\eta_1|\alpha_1^n}<c_{9}n^{k-1}\frac{1}{\alpha_1^{n-m}}.
\end{equation}
Similarly, to derive an upper bound for $\frac{a\sum_{j=2}^{t}|f_j(n)||\alpha_j|^n}{|z_r|}$, we may assume that $|\alpha_2| \geq \dots \geq |\alpha_t|$
\begin{equation}\label{eq:thm-17a}
\frac{a\sum_{j=2}^{t}|f_j(n)||\alpha_j|^n}{|z_r|}<\frac{2rtac_3n^{k-1}|\alpha_2|^n}{|\eta_1|\alpha_1^n}<c_{10}n^{k-1}\left(\frac{|\alpha_2|}{\alpha_1}
\right)^{n-m}.
\end{equation}
Finally using \eqref{eq:thm-14}, we get
\begin{equation}\label{eq:thm-15}
\frac{\sum_{i=1}^{r-1}{|z_i|}}{|z_r|} \leq \frac{r-1}{|z_r|^{\frac{\e}{1+\e}}}<\frac{(r-1)(2r)^{\frac{\e}{1+\e}}}{|\eta_1|^{\frac{\e}{1+\e}}\alpha_1^{\frac{\e}{1+\e}n}}\leq \frac{c_{11}}{\alpha_1^{\frac{\e}{1+\e}n}}\leq \frac{c_{11}}{\alpha_1^{\frac{\e}{1+\e}(n-m)}}.
\end{equation}
Thus \eqref{eq:thm-17}, \eqref{eq:thm-17a}, and \eqref{eq:thm-15} altogether imply
\begin{equation}\label{eq:Phi-1_upper}
|\Phi_1-1|<c_{12}n^{k-1}\left\{\max\left(\frac{1}{\alpha_1}, \frac{1}{\alpha_1^{\frac{\e}{1+\e}}},\frac{|\alpha_2|}{\alpha_1}
\right)\right\}^{n-m}.
\end{equation}

Now, we distinguish two cases according as $|\Phi_1-1| > \frac{1}{2}$ or $0 \leq |\Phi_1-1| \leq \frac{1}{2}$, respectively.

If $|\Phi_1-1| > \frac{1}{2}$, then using the fact that $\alpha_1$ is dominant in \eqref{eq:Phi-1_upper}, we obtain $n-m \leq c_{13}$.

So, we may assume that $0 \leq |\Phi_1-1| \leq \frac{1}{2}$. Suppose that $|\Phi_1-1|=0$, that is $\Phi_1=1$, we get by the definition of $\Phi_1$ that
\begin{equation}
\label{eq:thm-10}
a\eta_1\alpha_1^n=z_r.
\end{equation}
Since $\alpha_1 \not\in \mathbb{Z}$ then there exists a conjugate $\alpha_j$ of $\alpha_1$ in the field $\mathbb{K}=\mathbb{Q}(\alpha_1,\dots,\alpha_t)$ such that $\alpha_1 \ne \alpha_j$. Again $z_r$ is a rational integer, therefore, on taking the $j$-th conjugate of both sides of \eqref{eq:thm-10},
we may write that $af_j(n)\alpha_j^n=z_r$.
Thus
$$
\left(\frac{\alpha_1}{|\alpha_j|}\right)^n=\left|\frac{f_j(n)}{\eta_1}\right|,
$$
which by Lemma \ref{lem4}(ii) yields that $n \leq c_{14}$.

Finally, we may assume that $0 < |\Phi_1-1| \leq \frac{1}{2}$. It is well known that if $|\Phi_1-1| \leq \frac{1}{2}$ we have $|\log(\Phi_1)| < 2|\Phi_1-1|$, where $\log(\Phi_1)$ denotes the principal value of
the logarithm of the complex number $\Phi_1$. Hence, from \eqref{eq:Phi-1_upper} we conclude that
\begin{equation} \label{eq:lambda upper}
|\Lambda_1| < 2|\Phi_1-1|<2c_{12}n^{k-1}\left\{\max\left(\frac{1}{\alpha_1}, \frac{1}{\alpha_1^{\frac{\e}{1+\e}}},\frac{|\alpha_2|}{\alpha_1}\right)\right\}^{n-m}
\end{equation}
where $\Lambda_1:=\log(\Phi_1)$. Since $z_i \in U_{S}$ $(1 \leq i \leq r)$, we may write $|z_r|=p_1^{e_{r1}} \dots p_{\ell}^{e_{r\ell}}$, where the numbers $e_{rj}$ are nonnegative integers. Further, since $\Phi_1=a\eta_1\alpha_1^{n}z_r^{-1}$, we have
$$
\Lambda_1=(-1)\sum_{i=1}^{\ell}(e_{ri})\log(p_i)+\log(a\eta_1)+n\log(\alpha_1)+b_0\log(-1),
$$
where $b_0$ is an integer with $|b_0| \leq \ell+2$.

We use Lemma \ref{lem:matveev} to obtain a non-trivial lower bound for $|\Lambda_1|$. Since $\Phi_1 \ne 1$ we also have that $\Lambda_1 \ne 0$. Hence, we can apply Lemma \ref{lem:matveev} to $\Lambda_1$ with $\mathbb{K}=\mathbb{Q}(\alpha_1,\dots,\alpha_t), D \leq k^t, m=\ell+3, \psi_j=p_j \ (1 \leq j \leq \ell), \psi_{\ell+1}=a\eta_1, \psi_{\ell+2}=\alpha_1, \psi_{\ell+3}=-1$. By Lemma \ref{ubound}, we may choose $B=c_{15}n$ with an effective constant $c_{15}$. Also, it is clear that for $1 \leq j \leq \ell$ the choice $A_j=\log{p_j}$ is suitable. Further, by Lemma \ref{lem4}(ii) there exists an effective constant $c_{16}$ such that
$$
\max\{A_{\ell+1}, A_{\ell+2}, A_{\ell+3}  \} \leq c_{16}\log n.
$$
Applying Lemma \ref{lem:matveev} with the above parameters we obtain 
\begin{equation}\label{lambda lower}
|\Lambda_1|>\exp(-c_{17}\log{n}).
\end{equation}
Comparing \eqref{lambda lower} with \eqref{eq:lambda upper}, we get that $n-m \leq c_{18}\log n$. 
\end{proof}

\section{Proof of Theorem \ref{th1}}

Let $c_1, c_2, \dots$ be positive numbers which are effectively computable  in terms of $\e, \gamma, r, \ell$, $p_1,\dots, p_\ell, k, a, b$. Without loss of generality we may assume $n\geq m$. In view of  Lemma \ref{ubound} the upper bound for $\max(\log |z_1|, \dots, \log|z_r|)$ is $n$. Therefore, it is enough to derive an effective upper bound only for $n$.  In the case of $n = 0$, we have done immediately. So in the rest of the proof, we will assume that $n>0$.

Let $\alpha_1$ be the dominant root of $(U_n)_{n\geq 0}$ with $\eta_1\neq 0$ in \eqref{eq7}. Firstly, let $f_i(n)=0$ for some $2 \leq i \leq t$. Then Lemma \ref{lem4}(i) gives that $n \leq c_{1}$. Next, assume that 
\[|aU_n+bU_m| \leq \frac{1}{2}|\eta_1|\alpha_1^n\] and $f_i(n) \ne 0$ for $2 \leq i \leq t$. Since by assumption $|aU_n + bU_m| \geq |U_n|$, we get
\[|U_n|\leq |aU_n + bU_m| \leq   \frac{1}{2}|\eta_1|\alpha_1^n.\]
Now, in view of \eqref{eq4}, 
$$
\frac{1}{2}|\eta_1|\alpha_1^n \leq \left|\sum_{i=2}^tf_i(n)\alpha_i^n\right| \leq \sum_{i=2}^t|f_i(n)||\alpha_i^n|.
$$
Since $\alpha_1$ is the dominant root, Lemma \ref{lem4}(ii) gives $n \leq c_{2}$. 
Henceforth, we may assume that
\begin{equation*}
 f_i(n) \ne 0, \ (2 \leq i \leq t) \ \textrm{and} \ |aU_n+bU_m| > \frac{1}{2}|\eta_1|\alpha_1^n.
\end{equation*}
 Now, we may rewrite \eqref{eq8} as 
$$
a\eta_1\alpha_1^{n} + b\eta_1 \alpha_1^m - z_r  = \sum_{i=1}^{r-1}{z_i} - a\sum_{j=2}^{t}f_j(n)\alpha_j^n -b\sum_{j=2}^{t}f_j(m)\alpha_j^m,
$$
and from this we obtain
\[\eta_1\alpha_1^n(a+b\alpha_1^{m-n})z_r^{-1} -1= \frac{\sum_{i=1}^{r-1}{z_i} }{z_r} -  \frac{a \sum_{j=2}^{t}f_j(n)\alpha_j^n}{z_r} -  \frac{b\sum_{j=2}^{t}f_j(m)\alpha_j^m }{z_r}.\]
Hence by substituting $\Phi_2:=\eta_1\alpha_1^n(a+b\alpha_1^{m-n})z_r^{-1}$ in the above equation, we obtain
\begin{align}\label{eq:thm-9a}
\begin{split}
\left| \Phi_2 -1 \right| &\leq \frac{\sum_{i=1}^{r-1}{|z_i|}}{|z_r|}+\frac{a\sum_{j=2}^{t}|f_j(n)||\alpha_j|^n}{|z_r|} + \frac{b\sum_{j=2}^{t}|f_j(m)||\alpha_j|^m}{|z_r|}\\
&\leq \frac{2r\sum_{i=1}^{r-1}{|z_i|}}{|\eta_1|\alpha_1^{n}}+\frac{2ra\sum_{j=2}^{t}|f_j(n)||\alpha_j|^n}{|\eta_1|\alpha_1^{n}} + \frac{2rb\sum_{j=2}^{t}|f_j(m)||\alpha_j|^m}{|\eta_1|\alpha_1^{n}}\\
&\leq c_{3}n^{k-1}\left\{\max\left(\frac{1}{\alpha_1^{\frac{\e}{1+\e}}},\frac{|\alpha_2|}{\alpha_1}
\right)\right\}^{n}.
\end{split}
\end{align}
One can easily see that $n$ is bounded if $|\Phi_2-1|> \frac{1}{2}$. Thus, we may assume that $0\leq |\Phi_2-1|\leq \frac{1}{2}$. If  $|\Phi_2-1| = 0$, then we have
\begin{equation}\label{eq:thm-9b}
\eta_1(a\alpha_1^n+b\alpha_1^{m}) = z_r.
\end{equation}
Since $\alpha_1\notin \mathbb{Z}$, then there exists a conjugate $\alpha_j$ of $\alpha_1$ in the field $\mathbb{Q}(\alpha_1, \ldots, \alpha_t)$ such that $\alpha_1\neq \alpha_j$. By taking $j$-th conjugate of \eqref{eq:thm-9b}, we get
\begin{equation}\label{eq:thm-9c}
\eta_1^{(j)}\cdot \left(a\alpha_j^{n}+ b\al_j^{m}\right) = z_r,
\end{equation}
 (where $\eta_1^{(j)}$ is $j$-th conjugate of $\eta_1$).  As $\al_1>1$ is real and $a, b$ are positive integers, it follows from \eqref{eq:thm-9b} and \eqref{eq:thm-9c} that,
\begin{align*}
|\eta_1|\alpha_1^{n} & < |\eta_1|(a\alpha_1^{n}  + b\alpha_1^{m}) = | \eta_1^{(j)}||a\alpha_j^{n}+ b\al_j^{m}| \leq \max\left\{2A|\eta_1^{(j)}|, 2A|\eta_1^{(j)}||\alpha_j|^m\right\},
\end{align*}
where $A = \max\{a, b\}$. This implies that $n$ is bounded by an effectively computable constant $C$. Hence, we may assume that $0 < |\Phi_2-1| \leq \frac{1}{2}$. From \eqref{eq:thm-9a}, we infer that
\begin{equation}\label{eq:thm-9d}
|\Lambda_4| < 2|\Phi_2-1|=2c_{3}n^{k-1}\left\{\max\left(\frac{1}{\alpha_1^{\frac{\e}{1+\e}}},\frac{|\alpha_2|}{\alpha_1}\right)\right\}^{n}
\end{equation}
where $\Lambda_4:=\log(\Phi_2) = \log |\eta_1\alpha_1^n(a+b\alpha_1^{m-n})z_r^{-1}|$. Since $z_r\in \mathcal{U}_S$, 
$$
\Lambda_4=\sum_{i=1}^{\ell}(-e_{ri})\log(p_i)+\log|\eta_1|+n\log(\alpha_1)+ \log|a+b\alpha_1^{m-n}|+b_0\log(-1),
$$
where $b_0$ is an integer with $|b_0| \leq \ell+3$.

To apply Lemma \ref{lem:matveev}, we set $\mathbb{K}=\mathbb{Q}(\alpha_1,\dots,\alpha_t), D \leq k^t, m=\ell+4, \psi_j=p_j \ (1 \leq j \leq \ell), \psi_{\ell+1}=\eta_1, \psi_{\ell+2}=\alpha_1, \psi_{\ell+3} = a+b\alpha_1^{m-n}, \psi_{\ell+4}=-1$. By Lemma \ref{ubound}, we may choose $B=c_{4}n$ with an effective constant $c_{4}$. Now, it is clear that for $1 \leq j \leq \ell$ the choice $A_j=\log{p_j}$ is suitable. By Lemma \ref{lem4} and Proposition \ref{prop1}, there exists an effective constant $c_{5}$ such that
$$
\max\{A_{\ell+1}, A_{\ell+2}, A_{\ell+3}, A_{\ell+4} \} \leq c_{5}\log n.
$$
Applying Lemma \ref{lem:matveev} with the above parameters, we find
\begin{equation}\label{eq:thm-9e}
|\Lambda_4|>\exp(-c_{6}\log{n}).
\end{equation}
Comparison of \eqref{eq:thm-9d} with \eqref{eq:thm-9e} yields $n \leq c_{7}\log n$. This completes the proof of the Theorem \ref{th1}. \qed

{\bf Data Availability Statements} Data sharing not applicable to this article as no datasets were generated or analysed during the current study.

{\bf Acknowledgment:} The authors sincerely thank the referee for his/her thorough reviews and very helpful comments and suggestions which significantly improves the paper. The first author's work is supported by CSIR fellowship(File no: 09/983(0036)/2019-EMR-I).

\end{document}